\documentclass[12pt]{amsart}
\usepackage{amsmath,amsthm,mathrsfs,amssymb}

\newcommand{\N}{\mathbb{N}}
\newcommand{\R}{\mathbb{R}}
\newcommand{\Z}{\mathbb{Z}}

\newcommand{\T}{\mathcal{T}}
\newcommand{\den}{\mathcal{D}}
\newcommand{\myrare}{\smash[b]{\underline{{\mathcal{R}}}}_r(\xi)}
\newcommand{\myden}{\smash[b]{\underline{{\mathcal{D}}}}_r(\xi)}

\newcommand{\essinf}{\mathrm{ess\hspace{1mm}inf}}
\newcommand{\esssup}{\mathrm{ess\hspace{1mm}sup}}
\newcommand{\set}[1]{\left\{#1\right\}}
\newcommand{\e}{\varepsilon}
\renewcommand{\S}{\mathcal{S}}
\renewcommand{\P}{\mathbb{P}_{\theta}}
\newcommand{\E}{\mathbb{E}_{\theta}}
\newcommand{\sol}{v_{\xi}(t,x; x_0)}
\newcommand{\mass}{v_{\xi}(t; x_0)}
\renewcommand{\i}{i \mspace{-3mu} \imath}

 \newtheorem{thm}{Theorem}[section]
 
 \newtheorem{lem}[thm]{Lemma}
 \newtheorem{prop}[thm]{Proposition}
 \theoremstyle{definition}
 \newtheorem{Def}[thm]{Definition}
 \theoremstyle{remark}
 \newtheorem{rem}[thm]{Remark}
 
 \numberwithin{equation}{section}
\newtheorem*{spec}{Spectral control}{\bf}{\it}

\begin{document}

\title[Parabolic Anderson problem with  
a perturbed lattice potential]
{Moment asymptotics for the parabolic Anderson problem with  
a perturbed lattice potential}

%----------Author 1
\author[R. Fukushima]{Ryoki Fukushima}

\address{%
Department of Mathematics\\ 
Kyoto University\\ 
Kyoto 606-8502\\ 
JAPAN\\
Current address:\\
Department of Mathematics\\
Tokyo Institute of Technology\\ 
Tokyo 152-8551\\
 JAPAN
}

\email{ryoki@math.titech.ac.jp}

\thanks{The first author was partially supported by JSPS Fellowships for Young Scientists.\\
The second author was partially supported by  KAKENHI (21540175)}
%----------Author 2
\author[N. Ueki]{Naomasa Ueki}
\address{Graduate School of Human and Environmental Studies\\
Kyoto University\\
Kyoto 606-8501\\ 
JAPAN} 
\email{ueki@math.h.kyoto-u.ac.jp}

\begin{abstract}
The parabolic Anderson problem with a random potential obtained by attaching 
a long tailed potential around a randomly perturbed lattice 
is studied. 
The moment asymptotics of the total mass of the solution is derived.
The results show that the total mass 
of the solution concentrates on a small set in the space of configuration. 
\end{abstract}
 \maketitle
%%% Introduction %%%
\section{Introduction}\label{intro}
This paper is a continuation of \cite{FU09a}. 
We consider the initial value problem of the heat equation with a random potential
\begin{equation}
\begin{aligned}
	\frac{\partial}{\partial t}v(t,x) &= \frac{1}{2} \Delta v(t,x) - V_{\xi}(x) v(t,x), \quad 
	&&(t, x) \in (0, \infty) \times \R^d, \\
	v(0,x) &= \delta_{x_0}(x), && x \in \R^d, \label{PAM}
\end{aligned}
\end{equation}
where $\Delta$ is the Laplacian, $x_0 \in \R^d$, and
\begin{equation}
	V_{\xi}(x) := \sum_{q \in \Z^d} u(x-q-\xi_q) \label{potential}
\end{equation}
with $\xi =(\xi_q)_{q \in \Z^d}$ a collection of independent and identically 
distributed random vectors. 
Under appropriate assumptions, \eqref{PAM} has a solution $\sol$ 
represented by the Feynman-Kac formula 
\begin{equation}
 \sol
	=E_{x_0}\biggl[\exp\set{-\int_0^tV_{\xi}(B_s)ds}\biggl| B_t=x \biggr] 
	 \frac1{(2\pi t)^{d/2}}\exp \Big( -\frac{|x-x_0|^2}{2t}\Big),\label{solution}
\end{equation}
where $(B_s)_{s\ge 0}$ is the Brownian motion on $\R^d$ and $E_{x_0}$ is 
the expectation of the Brownian motion starting at $x_0$.

In this paper, we investigate the long time asymptotics of the moment of the total mass
\begin{equation}
   \mass :=\int_{\R^d}\sol dx_0
	=E_{x_0}\biggl[\exp\set{-\int_0^tV_{\xi}(B_s)ds}\biggr].\label{total mass}
\end{equation}
Our main result is Theorem \ref{leading term simplified}, which deals with the first moment.
We also obtain results on the higher moments in Section \ref{int} below.

The operator $H_{\xi}=-\Delta /2+V_{\xi}$ is the Hamiltonian of the so-called  
random displacement model in the theory of random Schr\"odinger operators and 
there has recently been an increase in research, see e.g.~\cite{BLS07, BLS08, Fuk09a, FU09a, GK09}.  
Also, the initial value problem~\eqref{PAM} itself is called the ``parabolic Anderson problem'' 
in literature (see e.g.~a survey article by G\"artner and K\"onig~\cite{GK04}).
The solution of the parabolic Anderson problem is believed to concentrate on a relatively small region
and there are many results support this concentration. 
We shall discuss this aspect in more detail in Subsection~\ref{Intermittency} below. 

%%% Basic assumptions %%%
\subsection{Basic assumptions}
We are mainly interested in the case where the single site potential and the displacement 
variables satisfy the following: 
(i) $u$ is a nonnegative function belonging to the Kato class $K_d$ (cf.~\cite{Szn98}) and 
\begin{equation}
	u(x) =C_0|x|^{-\alpha}(1+o(1))\label{single-decay}
\end{equation} 
as $|x|\to \infty$ for some $\alpha >d$ and $C_0>0$; (ii) each $\xi_q$ has the explicit 
distribution
\begin{equation}
	\P(\xi_q \in dx) = \frac{1}{Z(d,\theta)} \sum_{p \in \Z^d} 
	\exp (-|p|^{\theta} )\delta_p(dx) \label{disc-dist}
\end{equation}
for some $\theta >0$ and the normalizing constant $Z(d,\theta)$.
 
We also consider the case that $u$ is a nonpositive function. 
For this case, we assume 
$\inf u=u(0)>-\infty$, \eqref{single-decay} for some $C_0<0$, and that for any 
$\e >0$, there exists $R_{\e}>0$ 
such that $u(x) \le u(0)+\e$ for $|x|<R_{\e}$. 
Nevertheless, our main interest is the nonnegative case and we assume $u \ge 0$ 
unless otherwise specified. 

%%% Motivation %%%
\subsection{Motivation}
In Theorem~6.3 of the preceding paper~\cite{FU09a},
we have shown the following:

\begin{thm}\label{F-K}
Let us define 
\begin{equation}
 c(d, \alpha, \theta, C_0):= \int_{\R^d}dq \inf_{y\in \R^d}( \frac{C_0}{|q+y|^{\alpha}}+|y|^{\theta}).\label{const}
\end{equation} 
\begin{enumerate}
\item{Assume that $d=1$ and that $\essinf _{B(R)}u>0$ for any $R\ge 1$ if $\alpha \le 3$.
Then we have 
\begin{equation}
 \log \E[\mass ] 
  \begin{cases}
   \sim -t^{(1+\theta )/(\alpha+\theta )}c(1, \alpha, \theta, C_0)&(1<\alpha<3)\\[8pt]
   \asymp -t^{(1+\theta )/(3+\theta )} & (\alpha = 3), \\[5pt]
   \displaystyle 
   \sim {-t^{(1+\theta )/(3+\theta )}\frac{3+\theta}{1+\theta}\big( \frac{\pi ^2}8 \big) ^{(1+\theta )/(3+\theta )}}
   & (\alpha >3)
 \end{cases}\label{F-K-1}
\end{equation}
as $t \to \infty$, where $f(t) \sim g(t)$ means $\lim_{t\to \infty}f(t)/g(t)=1$
and $f(t) \asymp g(t)$ means $0<\varliminf _{t\to \infty}f(t)/g(t)\le \varlimsup _{t\to \infty}f(t)/g(t)<\infty$.\\
}
\item{Assume that $d=2$ and that $\essinf _{B(R)}u>0$ for any $R\ge 1$ if $\alpha \le 4$.
Then we have  
\begin{equation}
 \log \E[\mass ]  
  \begin{cases}
   \sim -t^{(2+\theta )/(\alpha+\theta )}
   c(2, \alpha, \theta, C_0) & (2 < \alpha < 4), \\[8pt]
   \asymp -t^{(2+\theta )/(4+\theta )} & (\alpha = 4), \\[5pt]
   \asymp {-t^{( 2+\theta )/(4+\theta )} (\log t)^{-\theta /(4+\theta )}} & (\alpha >4)
 \end{cases}\label{F-K-2}
\end{equation}
as $t \to \infty$.\\ 
}
\item{Assume that $d \ge 3$ and that $\essinf _{B(R)}u>0$ for any $R\ge 1$ if $\alpha \le d+2$.
Then we have 
\begin{equation}
 \log \E[\mass ] 
  \begin{cases}
   \sim -t^{(d+\theta )/(\alpha+\theta )}
   c(d, \alpha, \theta, C_0) & (d < \alpha < d+2), \\[10pt]
   \asymp {-t^{(d+\theta\mu )/(d+2+\theta\mu )}} & (\alpha \ge d+2)
 \end{cases} \label{F-K-3}
\end{equation}
as $t \to \infty$, where 
\begin{equation}
\mu =\frac{2(\alpha -2)}{d(\alpha -d)}. \label{F-K-mu}
\end{equation}\\ }
\item{Assume $u \le 0$, $\sup u=u(0)>-\infty$, and the existence of $R_{\e}>0$ for any $\e >0$
such that $\esssup _{B(R_{\e})}u\le u(0)+\e$. 
Then we have 
\begin{equation}
 \log \E[\mass ] \sim t^{1+d/\theta} c_-(d, \theta, u(0)) \label{F-K-4}
\end{equation}
as $t \to \infty$, where
\begin{equation}
 c_-(d, \theta, K):=\frac{2\pi ^{d/2}\theta |K|^{1+d/\theta}}{d(d+\theta )\Gamma (d/2)}\label{const-n}
\end{equation} 
for $K\in {\mathbb R}$. }
\end{enumerate}
\end{thm}
We have precise forms of the leading terms for the one-dimensional case with $\alpha \ne 3$, the general dimensional case with $d< \alpha <d+2$, and the case of $u \le 0$. Furthermore, if one goes into the proof of these results, it will be observed that 
only a very small set in $\xi$-space contributes the leading terms of the asymptotics. 
More precisely, when $u \ge 0$ and $d<\alpha<d+2$ for instance, 
the $y$-variable in the definition of $c(d, \alpha, \theta, C_0)$ corresponds to 
the displacement $\xi_q$ from $q$. 
Therefore taking the infimum in the definition of $c(d, \alpha, \theta, C_0)$ with respect to $y$ means minimizing the sum of the 
contribution of $u(-q-\xi_q)$ to $V_{\xi}(0)$ and the cost for displacement for each $q$. 
With these interpretation, the above theorem says that only the \emph{optimal} configuration contributes the leading term. 
This kind of concentration in $\xi$-space is sometimes regarded as a collateral evidence 
of the aforementioned spatial irregularity of $\sol$, see Sect.~1.3 of~\cite{GK04}. 
The aim of this paper is to find a variational expression for the leading part in the remaining cases 
to see a concentration phenomenon similar to above. 

%%% Main result %%%
\subsection{Main result}\label{MEO}
We need to introduce some notations to state the results. 
We write $\Lambda_r$ for $[-r/2,r/2]^d$ and introduce scaling factors 
\begin{equation}
	r=
	\begin{cases}
	t^{1/(3+\theta)}&(d=1\textrm{ and }\alpha = 3),\\
	t^{1/(4+\theta)}(\log t)^{\theta/(8+2\theta)} & (d=2\textrm{ and }\alpha > 4), \\
	t^{1/(d+2+\mu\theta)} & (d \ge 3\textrm{ and }\alpha\ge d+2\textrm{ or }(d,\alpha )=(2,4)).
	\end{cases}\label{scale}
\end{equation}
For any open set $U$ and $\xi =(\xi _q)_{q\in \Z ^d}\in (\Z^d)^{\Z^d}$, 
we denote by $\lambda_{\xi}^r(U)$ the bottom of the spectrum of 
$$
 -\frac{1}{2}\Delta + V^r_{\xi}
$$
in $U$ with the Dirichlet boundary condition, where
$$
V^r_{\xi}(x):=\sum_{q \in \Z^d} r^2 u(rx-q-\xi_q).
$$ 
Finally, let $\Omega_t=(\Z^d)^{\Lambda_t \cap \Z^d}$, which is the set of possible configurations
of $(\xi_q)_{q \in \Lambda_t \cap \Z^d}$, and we write $\lambda_{\xi}^r(U)$ for the same 
object as above also for $\xi \in \Omega_t$ with the potential replaced by
$$
V^r_{\xi}(x):=\sum_{q \in \Z^d\cap \Lambda _t} r^2 u(rx-q-\xi_q).
$$ 

\begin{thm}\label{leading term simplified}
Assume that $\alpha =3$ for $d=1$ and $\alpha \ge d+2$ for $d\ge 2$. 
Under the above setting, we have
\begin{equation}
 \log \E[\mass ]\\
 = - tr^{-2}\inf_{\zeta \in \Omega_t}\Biggl\{\lambda_{\zeta}^r(\Lambda_{t/r})
 +\gamma (r)^{\theta}\sum_{q \in \Lambda_t \cap \Z ^d} r^{-d}\Bigl|\frac{\zeta_q}{r}\Bigr|^{\theta}\Biggr\}
 (1+o(1))
 \label{F-K upper simplified}
\end{equation}
as $t$ goes to $\infty$, where 
\begin{equation}
  \gamma (r)=
 \begin{cases}
 	1 & (d=1\textrm{ and }\alpha=3),\\
	\sqrt{(4+\theta )\log r} &(d=2\textrm{ and }\alpha >4),\\
	r^{1-\mu}  &(d \ge 3\textrm{ or }(d,\alpha )=(2,4)),
 \end{cases} \label{ccr}
\end{equation}
and $\mu$ is the number defined in \eqref{F-K-mu}.
\end{thm}
The interpretation of this result is as follows. 
For a given configuration $\xi = \zeta$, 
%the behavior of the total mass is governed by the bottom of the spectrum as 
the eigenfunction expansion indicates that 
\begin{equation}
	v_{\zeta}(t,x)=\exp\set{-\lambda_{\zeta}^1(\Lambda_{t})t(1+o(1))}\label{for fixed}
\end{equation}
since the contribution from outside $\Lambda_{t}$ is negligible. 
On the other hand, the probability to have such a configuration is formally given by 
\begin{equation}
	\P(\xi = \zeta)=\exp\Biggl\{-\sum_{q \in \Z^d}|\zeta_q|^{\theta}(1+o(1))\Biggr\}.
\end{equation}
Therefore, the variational problem to minimize the sum of the decay rate for fixed configuration 
and the cost to realize it has the form 
\begin{equation}
	\inf_{\zeta}\Biggl\{\lambda_{\zeta}^1(\Lambda_{t})t
	+\sum_{q \in \Z ^d}|\zeta_q|^{\theta}\Biggr\}, 
\end{equation}
which becomes almost the same as the right hand side of~\eqref{F-K upper simplified} 
after the scaling. Hence, the above theorem says that only the \emph{optimal} configuration 
contributes the leading part of the asymptotics, just as in the heavy tailed case. 

%%% Proof of Theorem 1.2 %%%
\section{Proof of Theorem \ref{leading term simplified}}
In Theorem 2.9 of \cite{Fuk09a}, the leading term for $\log \E[\mass ]$ 
with compactly supported $u$ was investigated by using Sznitman's ``method of enlargement of obstacles''. 
We shall apply the same method here.

%%% MEO %%%
\subsection{Method of enlargement of obstacles for the multidimensional case}
Let us first recall the elements of the methods developed in~\cite{Fuk09a}. 
It is basically a coarse graining method to establish a certain variational principle 
by reducing the number of configurations contributing the asymptotics. 
In this subsection, we define a set of reduced configurations and show that its cardinality is indeed negligible compared with the decay of $\E[\mass ]$
(see \eqref{conf-set} and \eqref{s-number} below). 

We take $\chi \in ((\mu -2/d)\theta ,\mu \theta )$ and $\eta \in (0,1)$ so small that 
$$\chi > \left(\mu - \frac{2}{d}\right)\theta +2\eta^2+\left({d-2}+\frac{2\theta}{d} \right)\eta$$
and define
$$\gamma := \frac{d-2}{d}+\frac{2\eta}{d}<1.$$
We further introduce a notation concerning a diadic decomposition of $\R^d$. 
For each $k\in \Z _+$, let $\mathcal{I}_k$ be the collection of indices 
$\i = (i_0, i_1, \ldots , i_k)$ with $i_0 \in \Z^d$ and $i_1, \ldots , i_k \in \set{0, 1}^d$. 
For each $\i \in \mathcal{I}_k$, we associate the box 
$$C_{\i} = q_{\i} + 2^{-k}[0,1]^d,$$
where
$$q_{\i} = i_0 + 2^{-1}i_1 + \cdots +2^{-k}i_k.$$ 
For $\i \in \mathcal{I}_k$ and $\i' \in \mathcal{I}_{k'}$ with $k' \le k$, 
$\i \preceq \i'$ means that the first $k'$ coordinates coincide. 
Finally, we introduce 
$$ n_{\beta} = \left[ \beta \,\frac{\log r}{\log 2} \right]$$
for $\beta > 0$ so that $2^{-n_{\beta}-1} < r^{-\beta} \le 2^{-n_{\beta}}$. 

We can now define the \underline{\smash[b]{density}} set, which we can discard 
from the consideration. 

\begin{Def}\label{def1}
 We call a unit cube $C_q$ with $q \in \Z^d$ a \underline{\smash[b]{density}} box if all 
 $q \preceq \i \in \mathcal{I}_{n_{\eta\gamma}}$ satisfy the following: for at least half of $\i \preceq \i' \in \mathcal{I}_{n_{\gamma}}$,
 \begin{equation}
  (q_{\i'} + 2^{-n_{\gamma}-1}[0,1]^d)\cap \{ (q+\xi _q)/r : q\in \Z ^d\} \ne \emptyset .\label{myden}
 \end{equation}
 The union of all \underline{\smash[b]{density}} boxes is denoted by $\smash[b]{\underline{\den}}_r(\xi)$. 
\end{Def}

The following theorem tells us that we can replace $\smash[b]{\underline{\den}}_r(\xi)$ by a hard trap
without causing a substantial increase in the principal eigenvalue. 

\begin{spec}  
There exists $\rho>0$ such that for all $M>0$ and sufficiently large $r$, 
\begin{equation}
 \sup_{\xi \in (\R ^d)^{\Z ^d}}
 \left(\lambda^{r}_{\xi}\left( \myrare \right)\wedge M 
 -\lambda^{r}_{\xi}\left( \Lambda _{t/r}  \right)\wedge M \right) \le r^{-\rho},\label{spec}
\end{equation}
where $\myrare=\Lambda _{t/r} \setminus \myden$. 
\end{spec}

By Proposition 2.7 in \cite{Fuk09a}, the proof of this theorem is reduced to the extension of Theorem 4.2.3 in \cite{Szn98} 
from the compactly supported single site potentials to the Kato class single site potentials, which is straightforward. 

For $\myrare$, we can give the following quantitative estimate on its volume: 

\begin{lem} \label{r-number}
{\upshape (i)} There exists a positive constant $c_1$ independent of $r$ such that
$|\myrare| \ge r^{\chi}$ implies
\begin{equation}
	 \sum_{q \in \Lambda_t\cap \Z^d} |\xi_q|^{\theta} \ge c_1 r^{d(1-\eta\gamma)+(1-\gamma)\theta+\chi}.
\end{equation}
{\upshape (ii)} There exists a positive constant $c_2$ independent of $r$ such that
\begin{equation}
	\P(|\myrare| \ge r^{\chi})
	\le \exp (-c_2 r^{d(1-\eta\gamma)+(1-\gamma)\theta+\chi}).\label{r-number1} 
\end{equation}
In particular, $\P(|\myrare| \ge r^{\chi}) = o(\E[\mass ])$.
\end{lem}

\begin{proof} 
Throughout the proof, $c_1$ and $c_2$ are positive constants whose values may change line by line. 
We consider the following necessary condition of $C_q \not\subset \myden$: 
\begin{equation}
 \begin{split}
  &\textrm{there exists an }\i \succeq q\; \textrm{in }\mathcal{I}_{n_{\eta\gamma}} \textrm{ such that for a half of }\i' \succeq \i \: \textrm{in }\mathcal{I}_{n_{\gamma}}, \\
  &\{ r^{-1}q'+r^{-1}\xi_{q'} : q'\in (rC_{\i'})\cap \Z^d\} \not\subset q_{\i'} + 2^{-n_{\gamma}-1}[0,1]^d. \label{suff}
 \end{split}
\end{equation}
Note first that 
$$ \sum_{q' \in (rC_{\i'})\cap \Z^d}|\xi_{q'}|^{\theta}
 \ge \sum_{q' \in (rC_{\i'})\cap \Z^d}|d(q', \partial (r C_{\i'}))|^{\theta}
 \ge c_1 r^{(1-\gamma)(d+\theta)}$$
for any configurations satisfying the second line in \eqref{suff}.
Thus $C_q \not\subset \myden$ implies 
\begin{equation}
\begin{split}
	\sum_{q' \in (r C_{q})\cap \Z^d}|\xi_{q'}|^{\theta}
	&\ge c_1 r^{(1-\gamma)(d+\theta)}2^{d(n_{\gamma}- n_{\eta \gamma})-1}\\
	&\ge c_2 r^{(1-\gamma)(d+\theta)+d\gamma(1-\eta)}\label{individual}
\end{split}
\end{equation}
and the first assertion follows from this. 

For the second assertion, we use \eqref{individual} and take the sum over the possibilities of the indices 
$\i$ and $\i'$'s in \eqref{suff} to obtain 
\begin{equation*}
 \begin{split}
  &\P(\textrm{\eqref{suff} is satisfied}) \\
  \le &\, 2^{d n_{\eta \gamma}} \binom{2^{d(n_{\gamma}- n_{\eta \gamma})}}{2^{d(n_{\gamma}- n_{\eta \gamma})-1}}
  \exp (-c_1 r^{(1-\gamma)(d+\theta)+d\gamma(1-\eta)})\\
  \le &\, \exp (-c_2 r^{d(1-\eta\gamma)+(1-\gamma)\theta}) 
 \end{split}
\end{equation*}
for large $r$. In the second line, the first factor represents the choice of the index $\i$ and  
the second factor the choice of the indices $\i'$'s. 
Since the variables $\{ \xi_{q'} : q'\in C_q\cap \Z ^d\}$ are independent in $q \in \Z^d$, we have 
\begin{equation*}
 \begin{split}
  \P (|\Lambda_{t/r} \setminus \myden| \ge r^{\chi})
  \le &\,t^{d r^{\chi}}\left(\exp (-c_2r^{d(1-\eta\gamma)+(1-\gamma)\theta})\right)^{r^{\chi}}\\ 
  \le &\,\exp (-c_2r^{d(1-\eta\gamma)+(1-\gamma)\theta+\chi}), 
 \end{split}
\end{equation*}
which is the desired estimate. 

Finally the third assertion follows from Theorem \ref{F-K} and our choice of $\chi$. \end{proof}

With the help of this lemma, we may restrict ourselves on some special 
configurations. To see this, we introduce some more notations. 
A domain $R$ is called a lattice animal if it is represented as
$$R=\left(\overline{\bigcup_{q\in S(R)}\Lambda _1(q)}\right)^{\circ},$$
where $S(R)\subset \Z ^d$ consists of adjacent sites. 
This means that $R$ is a combination of unit cubes connected via faces.
We set
\begin{equation}
 \begin{split}
  \S_r = \bigl\{&(R_r, \zeta =(\zeta_q)_{q \in (r[R_r: \,l])\cap \Z ^d}): 
  R_r \textrm{ is a lattice animal included in }\Lambda _{t/r}, \\
  &|R_r| < r^{\chi}, q+\zeta_q \in [\T : t^{1/(\mu \theta )}]\cap \Z ^d \textrm{ for all }q \in (r[R_r: \,l])\cap \Z ^d \bigr\}, 
 \end{split} \label{conf-set}
\end{equation}
where $l$ is a positive number specified later, 
and $[A:l]=\{ x\in \R ^d : d(x,A)<l \}$ for any $A\subset \R ^d$. 
For any $(R_r ,\zeta )\in \S_r$, we write
$$
 V^r_{\zeta}(x) = \sum_{q \in (r[R_r: \,l])\cap \Z ^d} r^2 u(rx-q-\zeta _q)
$$
with a slight abuse of the notation and define $\lambda_{\zeta}^r(R_r)$ accordingly. 

We now see that the \emph{relevant} configurations of $(\myrare, \xi)$ are only the pairs in $\S_r$.
In fact removing the points $\{ q+\xi_q: q \in \Z^d \setminus (r[R_r: l]) \}$, which should be cared in proving the lower bound, is permitted as we will show in Lemma \ref{outer} below.
We also have   
$$
 \lambda^{r}_{\xi}\left( \myrare \right)=\lambda^{r}_{\xi}\left( R_r \right)
$$
for some lattice animal $R_r$ included in $\myrare$ and 
$$\P (q+\xi_q \notin [\T : t^{1/(\mu \theta )}] \text{ for some }q \in (r[R_r: l])\cap \Z ^d)$$ 
decays exponentially in $t$. 
The latter easily follows by observing that 
$$d(r[R_r: l], [\T : t^{1/(\mu \theta )}]^c) > t^{1/\theta},$$
which is due to $lr + t^{1/\theta} <t^{1/(\mu \theta )}$, for large $t$. 

The key point in our coarse graining method is that 
the number of relevant configurations is estimated as 
\begin{equation}
\# \S_r \le t^{dr^{\chi}}(t+2t^{1/(\mu \theta )})^{dr^{d+\chi}c(1+l)}=o(\E[\mass ]^{-1}) \label{s-number}
\end{equation}
by an elementary counting argument, where $c$ is a finite constant depending only on $d$.
The second relation comes from our choice of $\chi$. 

%%% multidimensional case %%%
\subsection{Proof of a modified statement for the multidimensional case}
We state and prove slightly modified versions of Theorem~\ref{leading term simplified} 
in this section. They are shown to be equivalent to Theorem~\ref{leading term simplified}
in Subsection 2.4 below. Let us start with the multidimensional case. 

\begin{thm}\label{leading term}
Let $d \ge 2$ and assume  
the setting of Theorem~\ref{leading term simplified}. Then we have the following:
\begin{enumerate}
\item{For any $\e>0$ and $l > 0$, there exists $t_{\e ,\,l}>0$ such that 
\begin{equation}
\begin{split}
 &t^{-1}r^2\log \E[\mass ]\\
 &\quad \le -(1-\e)\inf_{(R_r,\, \zeta) \in \S_r}\Biggl\{\lambda_{\zeta}^r(R_r)
 +\gamma (r)^{\theta}\sum_{q \in (r[R_r: \,l])\cap \Z ^d} r^{-d}\Bigl|\frac{\zeta_q}{r}\Bigr|^{\theta}\Biggr\}
 \label{F-K upper}
\end{split}
\end{equation}
for any $t\ge t_{\e ,\,l}$, where $\gamma (r)$ is the function defined in \eqref{ccr}.}
\item{If $\alpha >d+2$, then for any $\e>0$ and $l>0$, there exists $t_{\e ,\,l}>0$ such that 
\begin{equation}
\begin{split}
 &t^{-1}r^2\log \E[\mass ]\\
 &\quad \ge -(1+\e)\inf_{(R_r,\, \zeta) \in \S_r}\Biggl\{\lambda_{\zeta}^r(R_r)
 +\gamma (r)^{\theta}\sum_{q \in (r[R_r: \,l])\cap \Z ^d} r^{-d}\Bigl|\frac{\zeta_q}{r}\Bigr|^{\theta}\Biggr\}
 \label{F-K lower}
\end{split}
\end{equation}
for any $t\ge t_{\e ,\,l}$.
}
\item{If $\alpha =d+2$, then for any $\e>0$, there exist $t_{\e}>0$ and $l_{\e}>0$ such that 
\begin{equation}
\begin{split}
 &t^{-1}r^2\log \E[\mass ]\\
 & \quad \ge -(1+\e)\inf_{(R_r,\, \zeta) \in \S_r}\Biggl\{\lambda_{\zeta}^r(R_r)
 +\gamma (r)^{\theta}\sum_{q \in (r[R_r: \,l])\cap \Z ^d} r^{-d}\Bigl|\frac{\zeta_q}{r}\Bigr|^{\theta}\Biggr\}
 \label{F-K lower2}
\end{split}
\end{equation}
for any $t\ge t_{\e}$ and $l\ge l_{\e}$.
}
\end{enumerate}
\end{thm}

\begin{proof}
We first prove the upper bound in (i).
By a standard Brownian estimate and scaling, we have
\begin{equation}
 \begin{split}
 & \E[\mass ] \\
 \le & \E \otimes E_{x_0}\left[\exp \set{- \int_0^tV_{\xi}(B_s)ds }:
 \sup_{0\le s\le t}|B_s|_{\infty} < \frac{t}2\right] + e^{-ct}\\ 
 \le & \E \otimes E_{x_0/r}\left[\exp \set{- \int_0^{tr^{-2}}V_{\xi}^r(B_s)ds }:
 \sup_{0\le s\le tr^{-2}} |B_s|_{\infty} < \frac{t}{2r}\right] + e^{-ct}.
 \label{upper-spec}
 \end{split}
\end{equation}
For any $\e \in (0,1)$, there exists a finite constant $c_{\e}$ depending only on $d$ and $\e$ such that
the first term of the right hand side is less than
$$
 c_{\e}\E \left[\exp \set{- (1-\e) \lambda_{\xi}^r(\Lambda _{t/r})tr^{-2} }\right]
$$
by (3.1.9) of \cite{Szn98}.
By the spectral control \eqref{spec}, Lemma \ref{r-number}, and \eqref{s-number}, this quantity is less than
\begin{equation*}
 \begin{split}
   &\, o(\E[\mass ]^{-1}) \sup_{(R_r,\, \zeta) \in \S_r} 
  \P(\xi_q = \zeta_q \textrm{ for all } q \in (r[R_r: l])\cap \Z ^d) \\
  &\times \exp \set{- (1-\e) (\lambda_{\zeta}^r(R_r)\wedge M-r^{-\rho})tr^{-2}}. 
 \end{split}
\end{equation*}
Thus, we have 
\begin{equation}
\begin{split}
 t^{-1}r^2\log \E[\mass ]
 \le &- (1-2\e ) \inf_{(R_r,\, \zeta) \in \S_r}\Biggl\{\lambda_{\zeta}^r(R_r)\wedge M -r^{-\rho} \\
 &+t^{-1}r^2\sum_{q \in (r[R_r: \,l])\cap \Z ^d} \left(|\zeta_q|^{\theta} + \log Z (d, \theta)\right)\Biggr\} 
\end{split} \label{F-K Upper'}
\end{equation}
for sufficiently large $t$.
We can drop $M$ and $r^{-\rho}$ from the right hand side since Theorem \ref{F-K} tells us 
that the left hand side is bounded from below. 
Moreover, we can also neglect $\log Z (d, \theta)$ since
\begin{equation}
\# ((r[R_r: l])\cap \Z ^d)\le cr^{d+\chi}=o(tr^{-2}). \label{q-sum}
\end{equation}
After removing the above three terms, \eqref{F-K Upper'} gives us the upper bound. 

We next proceed to the lower bound. 
We pick a pair $(R_r^*, \zeta^*)$ which attains the infimum in the right hand side of \eqref{F-K lower}.  
Then we have the following estimate for the $L^2$-normalized nonnegative eigenfunction $\phi^*$ 
corresponding to $\lambda_{\zeta^* }^r(R^*_r)$. 
\begin{lem}\label{lem-mass}
There exist $p^*\in (rR^*_r)\cap \Z ^d$ and $c_0>0$ such that 
$$\sup_{x \in \Lambda_{2/r}(p^*/r)}V_{\zeta^*}^r(x) \le c_0 r^{d+\chi+2}$$
and 
\begin{equation}
 \int_{\Lambda_{1/r}(p^*/r)}\phi^*(x) \, dx \ge \frac{1}{2\| \phi^* \|_{\infty}} r^{-d-\chi}. \label{mass}
\end{equation}
\end{lem}

\begin{proof}
We fix $1<r_0 <\infty$ so that 
$$\frac{C_0}{2|x|^{\alpha}} \le u(x) \le \frac{2C_0}{|x|^{\alpha}}$$ 
for all $|x| > r_0$ and
take $k \in \N$ satisfying $2^{-k-3}\le r_0/r < 2^{-k-2}$. We divide $R_r^*$ into subboxes of sidelength $2^{-k}$  as 
\begin{equation*}
 R_r^* = \bigcup_{\i \in \mathcal{I}^*} C_{\i} \quad \textrm{for some }\mathcal{I}^* \subset \mathcal{I}_k. 
\end{equation*}
Let $\mathcal{C}$ be the union of all boxes 
$C_{\i}$ in $R_r^*$ whose enlarged boxes $q_{\i}+2^{-k}[-1,2]^d$ intersect with 
$\{r^{-1}(q+\zeta^*_q): q \in (r[R_r^*: \,l])\cap \Z^d\}$.
Then it is easy to see that if $C_{\i} \subset \mathcal{C}$, 
there exist $a \in C_{\i}$ and $c_1 > 0$ for which $V_{\zeta ^*}^r \ge c_1r^2 1_{B(a, 1/r)}$. 
Thus, by using Lemma~3.5 in~\cite{FU09a}, which states 
\begin{equation}
 \inf \{ \lambda _1((-\Delta +1_{B(b,1)})_R^N) : b\in \Lambda _R\} \ge cR^{-d},
\end{equation}
and the scaling with the factor $r$, we have 
\begin{equation*}
 \inf_{\phi \in C^{\infty}(C_{\i})}\set{\frac{1}{\|\phi\|^2_2} \int_{C_{\i}} \Big( \frac{1}{2}|\nabla \phi (x)|^2 
 + V_{\zeta^*}^r(x) \phi(x)^2\Big) dx } 
 \ge c_2 r^2
\end{equation*}
for all $C_{\i} \subset \mathcal{C}$ and consequently 
\begin{equation*}
 c_2 r^2 \int_{\mathcal{C}} \phi^*(x)^2 dx
 \le \int_{\mathcal{C}} \Big( \frac{1}{2}|\nabla \phi^*|^2(x) + V_{\zeta^*}^r(x) \phi^*(x)^2\Big) dx. 
\end{equation*}
Since the right hand side is bounded from above by $\lambda_{\zeta^* }^r(R^*_r)$, 
it follows that 
$$\int_{\mathcal{C}} \phi^*(x)^2 dx \le c_3 r^{-2}.$$ 
This implies 
$$\int_{R_r^* \setminus \mathcal{C}} \phi^*(x)^2 dx \ge 1/2$$ 
for large $r$ and 
hence we can find a $\Lambda_{1/r}(p^*/r)$ in $R_r^*\setminus \mathcal{C}$ such that
\begin{equation*}
 \| \phi^* \|_{\infty} \int_{\Lambda_{1/r}(p^*/r)}\phi^*(x) \, dx \ge
 \int_{\Lambda_{1/r}(p^*/r)}\phi^*(x)^2 \, dx \ge \frac{1}{2} r^{-d-\chi}. 
\end{equation*}

Finally, we show the bound $\sup_{x \in \Lambda_{2/r}(p^*/r)}V_{\zeta^*}^r(x) \le c_0 r^{d+\chi+2}$. 
Note first that we have $\sup_{x \in \Lambda_{2/r}(p^*/r)}r^2 u(rx-q-\zeta^*_q) \le c_4 r^2$ 
for each $q$ since $R_r^*\setminus \mathcal{C}$ keeps the distance larger than $(r_0+1)/r$ from 
$\{r^{-1}(q+\zeta^*_q): q \in (r[R_r^*: \,l])\cap \Z^d\}$. 
Multiplying the total number of points 
$\# \{r^{-1}(q+\zeta^*_q: q \in (r[R_r^*: \,l])\cap \Z^d\} \le (2l+1)^d r^{d+\chi}$, 
we obtain the result. 
\end{proof}

We bound $\E[\mass ]$ from below by 
\begin{equation}
 \begin{split}
  & \P\left(\xi_q=\zeta^*_{p^*+q} \textrm{ for }q \in (r[R^*_r: \,l])\cap \Z ^d-p^*\right)\\
  & \times \P \left(\sup_{x \in (rR^*_r-p^*)\cup \Lambda _2}\sum_{q \in \Z ^d\setminus \{ (r[R^*_r: \,l])\cap \Z ^d-p^*\}} 
  u(x-q-\xi_q)  < \frac{c_1}{(rl)^{\alpha -d}}\right)\\
  & \times E_{x_0}\bigg[ \exp\set{-\int_0^t \sum_{q \in (r[R^*_r: \,l])\cap \Z ^d-p^*}u(B_s-q-\zeta _{p^*+q}^*)ds}: \\
  & \qquad B_s\in \Lambda _2\textrm{ for }0\le s\le 1, B_1\in \Lambda _1, B_s\in rR^*_r-p^*\textrm{ for }1\le s\le t \bigg]\\
  & \times \exp \left( -\frac{c_1t}{(rl)^{\alpha -d}}\right).
  \label{crude}
 \end{split}
\end{equation}
The first factor is greater than or equal to
$$\exp \left( -\sum_{q \in (r[R_r: \,l])\cap \Z ^d} |\zeta_q|^{\theta}-cr^{d+\chi}\right)$$
by the same argument using \eqref{q-sum} as for the upper bound.
The last factor is greater than $\exp (-\e tr^{-2})$ for sufficiently large $r$ if $\alpha > d+2$, and for
sufficiently large $r$ and $l$ if $\alpha = d+2$.
To bound the second factor we use the following: 
\begin{lem}\label{outer}
Let $\{ R_r : r\ge 1\}$ be a family of lattice animals satisfying $R_r \subset \Lambda _{t/r}$ and $|R_r| < r^{\chi}$.
Let $k, l>0$.
Then there exist $c_1, c_2, c_3>0$ independent of $R_r$ such that 
 \begin{equation}
  \P \left( \sup_{x \in [rR_r:k]}\sum_{q \in \Z ^d\setminus (r[R_r: \,l])} u(x-q-\xi_q) < c_1(rl)^{-\alpha+d}  \right) \ge c_2 \label{outer1}
 \end{equation}
 for any $r\ge c_3$.
\end{lem}

\begin{proof}

We consider the event
\begin{equation}
  \set{d(q+\xi_q, [rR_r:k]) \ge \frac{1}{2}d(q, [rR_r:k]) \textrm{ for all }q \in \Z ^d\setminus (r[R_r: l])}.\label{outer-event}
\end{equation}
On this event, we have 
\begin{equation*}
 \begin{split}
   & \sum_{q \in \Z ^d\setminus (r[R_r:\, l])} |x-q-\xi_q|^{-\alpha} \le \sum_{q \in \Z ^d\setminus (r[R_r:\, l])}\Bigl(\frac{2}{d(q, [rR_r:k])} \Bigr)^{\alpha}\\
  \le & c_4\sum_{q\in \Z ^d:\,d(q,rR_r)\ge rl}d(q, [rR_r:k])^{-\alpha} \le c_5(rl)^{-\alpha+d}
 \end{split}
\end{equation*}
for any $x\in [rR_r:k]$ and large $r$. 
By this estimate and the assumption $u(x) = C_0 |x|^{-\alpha}(1+o(1))$, we see that the event in \eqref{outer-event} implies the event in \eqref{outer1}. 
Since the inequality in \eqref{outer-event} is satisfied if
$$|\xi _q|\le d(q,[rR_r:k])/2\textrm{ for all }q \in \Z ^d\setminus (r[R_r: l]),$$ 
the probability of the event \eqref{outer-event} is greater than or equal to
\begin{equation}
 \prod_{q \in \Z ^d\setminus (r[R_r: l])}\left( 1-\frac{1}{Z (d, \theta)} 
  \sum_{y\in \Z ^d : |y| \ge d(q,\,[rR_r:k])/2} \exp (-|y|^{\theta}) \right).\label{outer-prob}
\end{equation}
It is easy to see that
$$\frac{1}{Z (d, \theta)} 
  \sum_{y\in \Z ^d : |y| \ge d(q,\,[rR_r:k])/2} \exp (-|y|^{\theta})
\le \exp (-c_6d(q,[rR_r:k])^{\theta})$$
and
$$\# \{q \in \Z^d : n \le d(q, [rR_r:k]) < n+1 \} \le c_7r^{\chi+d}n^{d-1}.$$
By using also an elementary inequality $(1-x)^p \ge 1-px$ for any $p \ge 1$ and $0<x<1$, the quantity in \eqref{outer-prob} is greater than or equal to
$$  \prod_{rl-k\le n \in \N}\left( 1- \exp (-c_6n^{\theta}) \right)^{c_7r^{\chi +d}n^{d-1}}
  \ge  \prod_{rl-k\le n \in \N}\left( 1- c_8r^{\chi +d}\exp (-c_9n^{\theta}) \right) .$$
Since the right hand side is a convergent infinite product, we conclude \eqref{outer1}. 
\end{proof}

It remains to bound the third factor in \eqref{crude}. 
We use the bound 
$$\sup_{x \in \Lambda_{2/r}(p^*/r)}V_{\zeta^*}^r(x) \le c_0 r^{d+\chi+2}$$ 
in Lemma \ref{lem-mass} for $0 \le s \le 1$ and the positivity of 
$$
\inf_{x,y\in \Lambda _1}\exp (\Delta ^D_2/2)(x,y), 
$$
where $\exp (t\Delta ^D_2/2)(x,y)$, $(t,x,y)\in (0,\infty )\times \Lambda _2\times \Lambda _2)$ is 
the integral kernel of the heat semigroup generated by the Dirichlet Laplacian on $\Lambda _2$ multiplied by $-1/2$.
Then, we can show that the third factor is greater than 
\begin{equation}
 r^d\exp (-c_0 r^{d+\chi})\int_{\Lambda _{1/r}}dy\int_{R^*_r-p^*/r}dz\exp (-(t-1)r^{-2}H^*)(y,z)
\label{hk-factor}
\end{equation}
for large $r$ by using a scaling, where $\exp (-tH^*)(x,y)$, $(t,x,y)\in (0,\infty )\times (R^*_r-p^*/r)\times (R^*_r-p^*/r))$ 
is the integral kernel of the heat semigroup generated by the Schr\"{o}dinger operator
$$H^*=-\Delta /2+\sum_{q\in (r[R^*_r:l])\cap \Z ^d-p^*}r^2u(rx-q-\zeta ^*_{p^*+q})$$
in $R^*_r-p^*/r$ with the Dirichlet boundary condition.
By \eqref{mass}, the integral in \eqref{hk-factor} is greater than or equal to
\begin{equation*}
 \begin{split}
  &\int_{\Lambda _{1/r}}dy\int_{R^*_r-p^*/r}dz\exp (-(t-1)r^{-2}H^*)(y,z)\frac{\phi ^*(z+p^*/r)}{\| \phi ^* \| _{\infty}}\\
  \ge & \exp (-(t-1)r^{-2}\lambda_{\zeta^*}^r(R^*_r))/(2\| \phi ^* \|^2 _{\infty}r^{d+\chi}). 
 \end{split}
\end{equation*}
Finally $\| \phi ^* \| _{\infty}$ is bounded since
$$\phi^*(y) = \exp (\lambda_{\zeta^*}^r(R^*_r))\int \exp (-H^0)(y,z)\phi^*(z) \,dz ,$$
$\| \exp (-H^0)(y,\cdot )\| _2 \le 1$, and $\lambda_{\zeta^*}^r(R^*_r)$ is bounded by Theorem \ref{F-K} 
and the upper bound in (i), where $\exp (-tH^0)(x,y)$, $(t,x,y)\in (0,\infty )\times R^*_r\times R^*_r$, 
is the integral kernel of the heat semigroup generated by the Schr\"{o}dinger operator $H^0=-\Delta /2+V^r_{\zeta ^*}$
in $R^*_r$ with the Dirichlet boundary condition.
By all these the lower bounds (ii) and (iii) are proven.
\end{proof}

%%% 1D result %%%
\subsection{Proof of a modified statement for the one-dimensional case}
We first fix a constant $M>0$ such that 
\begin{equation*}
 \begin{split}
  \P \left( \{q+\xi_q : q \in \Z\} \cap (0, Mt^{1/(3+\theta)}) =\emptyset \right) 
  & \le \exp \left\{ -c M^{1+\theta} t^{(1+\theta)/(3+\theta)} \right\}\\
  & = o(\E[\mass ]), 
 \end{split}
\end{equation*} 
which is possible in view of Theorem \ref{F-K}. 
We define the set $\S_r$ of relevant configurations by 
\begin{equation*}
 \begin{split}
  \S_r = \bigl\{ &((m,n), \zeta=(\zeta_q)_{q \in (m-l r, n+l r) \cap \Z }) \\
  & : m, n \in \Z, -t \le m < n \le t, n-m \le Mr, |\zeta_q| \le t^{1/\theta}, \\
  & \{q+\zeta_q : q \in (m-lr, n+lr) \cap \Z \} \cap (m,n) = \emptyset \bigr\} 
 \end{split}
\end{equation*} 
in this case. 
Now we can state the result. 
\begin{thm}\label{leading-1D}
 Let $d=1$ and assume the setting of Theorem~\ref{leading term simplified}. 
 Then, for any $\e>0$, there exist $t_{\e}>0$ and $l_{\e}>0$ such that 
\begin{equation*}
 \begin{split}
    & -(1+\e) \inf_{((m,n), \zeta) \in \S_r} 
   \biggl\{ \lambda_{\zeta}^r((m/r, n/r)) + \sum_{q \in (m-l r, n+l r) \cap \Z} r^{-1}\Bigl| 
   \frac{\zeta_q}{r} \Bigr|^{\theta} \biggr\} \\
   & \le t^{-(1+\theta)/(3+\theta)} \log \E[\mass ] \\
   & \le -(1-\e) \inf_{((m,n), \zeta) \in \S_r} 
   \biggl\{ \lambda_{\zeta}^r((m/r, n/r)) + \sum_{q \in (m-l r, n+l r) \cap \Z} r^{-1}\Bigl| 
   \frac{\zeta_q}{r} \Bigr|^{\theta} \biggr\} , 
 \end{split}
\end{equation*}
 for all $t > t_{\e}$ and $l > l_{\e}$. 
\end{thm}

\begin{proof}
We only prove the upper bound. After having it, the lower bound follows exactly in the same way 
as for Theorem~\ref{leading term}. 

We use a simple version of the method of enlargement of obstacles where $\gamma = 1$ and any $2^{-n_1}$-box 
containing a point of $\{r^{-1} (q+\xi_q) : q \in \Z\}$ is a density box. 
Such a box indeed satisfies the quantitative Wiener criterion (2.12) in page 152 of \cite{Szn98} 
since even a point has positive capacity when $d=1$ (cf.\ page 153 of \cite{Szn98}). 
Then, the spectral control \eqref{spec} implies that we can impose the Dirichlet boundary condition 
on each point in $\{r^{-1} (q+\xi_q) \}_{q \in \Z}$. 

Combining this observation with a standard Brownian estimate and (3.1.9) in \cite{Szn98}, we find 
\begin{equation*}
 \begin{split}
  \E[\mass ] &\le \E \left[ c\left(1+\left(\lambda_{\xi}^1 ( (-t,t) ) t \right)^{1/2} \right) 
  \exp\left\{-\lambda_{\xi}^1 ( (-t,t) ) t  \right\} \right] + e^{-ct} \\
  & \le c_{\e} \E \left[ \sup_{k} \exp\left\{ -(1-\e) \lambda_{\xi}^r \left( r^{-1} I_k \right) tr^{-2}  \right\} \right] + e^{-ct},\\
 \end{split}
\end{equation*}
where $\e$ is an arbitrary positive constant and 
 $\{ I_k \} _k$ are the random open intervals such that $\sum_k I_k=(-t,t)\setminus \{ q+\xi _q : q \in \Z\}$. 
By considering all possibilities of $I_k$, we can bound the $\E$-expectation in the right hand side by 
\begin{equation*}
 \begin{split}
 \sum_{m,n\in \Z:-t \le m < n \le t}
 \E \Big[ & \exp\left\{ -(1-\e) \lambda_{\xi}^r \left( (m/r, n/r) \right) tr^{-2}  \right\} \\ 
 & : \{q+\xi_q : q \in \Z\} \cap (m, n)=\emptyset \Big] .
 \end{split}
\end{equation*}
Note that we can discard $(m,n)$ whose interval $n-m > M r$ thanks to our choice of $M$. 
Hence, we can restrict our consideration on $\S_r$ and we can also show 
$\# \S_r = \exp\{ o(t^{(1+\theta)/(3+\theta)}) \}$ by an elementary counting argument. 
Now, we have 
\begin{equation*}
\begin{split}
 & \E[\mass ] \\
 \le & \sum_{((m,n), \zeta) \in \S_r} \exp\left\{ -(1-\e) \lambda_{\zeta}^r \left((m/r, n/r)\right) tr^{-2}  \right\} 
 \P \left( \xi_q=\zeta_q \textrm{ for all } q \right) \\
 & +o(\E[\mass ])\\
 \le & \exp\biggl\{-(1-2\e) t^{(1+\theta)/(3+\theta)} \\
 & \times \inf_{((m,n), \zeta) \in \S_r} 
 \biggl\{ \lambda_{\zeta}^r((m/r, n/r)) + \sum_{q \in (m-l r, n+l r) \cap \Z} r^{-1}\Bigl| 
 \frac{\zeta_q}{r} \Bigr|^{\theta} \biggr\}\biggr\}, 
\end{split}
\end{equation*}
which is the desired estimate. 
\end{proof}

%%% Rewriting the variational problem %%%
\subsection{Proof of Theorem \ref{leading term simplified}}
In this section, we complete the proof of Theorem \ref{leading term simplified} by simplifying the variational expression in Theorem~\ref{leading term}. 
We treat only the multidimensional case since the modification for the one-dimensional case is straightforward. 
Recall that $\Omega_{t}=(\Z^d)^{\Lambda_{t}\cap\Z^d}$ is the set of possible configurations 
of $\{\xi _q\}_{q \in \Lambda_{t}\cap \Z^d}$. 
We first show
\begin{equation}
\begin{split}
	& \inf_{(R_r,\, \zeta) \in \S_r}\Biggl\{\lambda_{\zeta}^r(R_r)
	+\gamma (r)^{\theta}\sum_{q \in (r[R_r: \,l])\cap \Z ^d} r^{-d}\Bigl|\frac{\zeta_q}{r}\Bigr|^{\theta}\Biggr\}\\
	& \quad \ge (1-\e) \inf_{\zeta \in \Omega_{t} }\Biggl\{\lambda_{\zeta}^r(\Lambda_{t/r})
	+\gamma (r)^{\theta}\sum_{q \in \Lambda_{t} \cap \Z ^d} r^{-d}\Bigl|\frac{\zeta_q}{r}\Bigr|^{\theta}\Biggr\}.
	\label{rewrite-1}
\end{split}
\end{equation}
for sufficiently large $t$ (and $l$) if $\alpha \in (d, d+2)$ (resp.~$\alpha=d+2$). 
Let $(R_r^*,\, \zeta^*)$ be a minimizer of the variational problem in the first line. 
We extend $\zeta^*$ to $\zeta^{**} \in \Omega_{t}$ by setting 
$\zeta^{**}_q=0$ for $q \in (\Lambda _t\setminus r[R_r^*: \,l])\cap \Z ^d$. 
Then, it is obvious that 
\begin{equation}
	\sum_{q \in (r[R_r^*: \,l])\cap \Z ^d} r^{-d}\Bigl|\frac{\zeta^*_q}{r}\Bigr|^{\theta}
	\ge \sum_{q \in \Lambda_{t} \cap \Z ^d} r^{-d}\Bigl|\frac{\zeta_q^{**}}{r}\Bigr|^{\theta}.
\end{equation}
Moreover, we can prove 
\begin{equation}
	\sup_{x \in rR^*_r} \sum_{q \in \Z ^d\setminus (r[R_r^*:\, l])} |x-q-\zeta^{**}_q|^{-\alpha} 
	\le c_1(rl)^{-\alpha+d}
\end{equation}
for this $\zeta^{**}$. Therefore, we have 
\begin{equation}
	\lambda_{\zeta^*}^r(R_r^*)+c_2r^{-\alpha+d+2}l^{-\alpha+d} \ge \lambda_{\zeta^{**}}^r(\Lambda_{t/r})
\end{equation}
and this yields~\eqref{rewrite-1}. 

We next show
\begin{equation}
\begin{split}
	& \inf_{(R_r,\, \zeta) \in \S_r}\Biggl\{\lambda_{\zeta}^r(R_r)
	+\gamma (r)^{\theta}\sum_{q \in (r[R_r: \,l])\cap \Z ^d} r^{-d}\Bigl|\frac{\zeta_q}{r}\Bigr|^{\theta}\Biggr\}\\
	& \quad \le (1+\e) \inf_{\zeta \in \Omega_{t} }\Biggl\{\lambda_{\zeta}^r(\Lambda_{t/r})
	+\gamma (r)^{\theta}\sum_{q \in \Lambda_{t} \cap \Z ^d} r^{-d}\Bigl|\frac{\zeta_q}{r}\Bigr|^{\theta}\Biggr\}
	\label{rewrite-2}
\end{split}
\end{equation}
for sufficiently large $t$.
It follows from Lemma~\ref{r-number} that if a sequence $\{ \zeta ^t\} _t$ of configurations satisfies $\zeta ^t\in \Omega_t$ and $|\smash[b]{\underline{{\mathcal{R}}}}_r(\zeta ^t)| \ge r^{\chi}$ for any $t$, then we have 
\begin{equation} 
	\gamma(r)^{\theta} \sum_{q \in \Lambda_t\cap \Z^d} 
	r^{-d}\Bigl|\frac{\zeta_q^t}{r}\Bigr|^{\theta} \longrightarrow \infty. \label{divergence}
\end{equation}
as $t \to \infty$. Thus if each $\zeta^t$ is a minimizer of the right-hand side 
of~\eqref{rewrite-2}, then we have  $|\smash[b]{\underline{{\mathcal{R}}}}_r(\zeta^t)| < r^{\chi}$ for large $t$. 
We may also assume that $q+\zeta^t_q \in [\T : t^{1/(\mu \theta )}]$ for all 
$q \in (r[\smash[b]{\underline{{\mathcal{R}}}}_r(\zeta^t): \,l])\cap \Z ^d$ since otherwise
~\eqref{divergence} holds. 
We here extend $\zeta ^t$ to $(r[\smash[b]{\underline{{\mathcal{R}}}}_r(\zeta^t): \,l])\cap \Z ^d$ by $\zeta ^t_q=0$ for $q\in (r[\smash[b]{\underline{{\mathcal{R}}}}_r(\zeta^t): \,l])\cap \Z ^d\setminus \Lambda _t$.
There exists a lattice animal $R_r^t$ in ${\underline{{\mathcal{R}}}}_r(\zeta^t)$ such that $\lambda _{\zeta ^t}^r({\underline{{\mathcal{R}}}}_r(\zeta^t))=\lambda _{\zeta ^t}^r(R_r^t)$.
Then it follows that $(\smash[b]R_r^t, 
(\zeta^t_q)_{q \in (r[\smash[b]R_r^t: \,l])\cap \Z ^d}) \in \S_r$ for sufficiently large $t$. 
Combining with Spectral control~\eqref{spec}, we obtain~\eqref{rewrite-2}. 

%%%%%%%%%%%%%%%%%%%%%%%%%%%%%%%%%%%%%%%%%%%%%%%%%%section8%%%%%%%%%%%%%%%%%%%%%%%%%%%%%%%%%%%%%%%%%%%%%%%%%%%%%%%%%%%%%%%%%%%%%%%%%%
\section{Asymptotics of higher moments}\label{int}
In \cite{Fuk09a}, a result on the asymptotics for higher moments of the survival probability 
is shown as an application of the precise form of the leading term. 
We shall extend the result to our cases in this section. 
Our objects are the $p$-th moments $\E[\mass ^p]$ for $p \ge 1$.
We consider their asymptotics in Subsection~\ref{lifshitz-int}.
In Subsection \ref{Intermittency}, we discuss a related quantitative estimate on intermittency 
for the parabolic Anderson problem.
 
%%%%%%%%%%%%%%%%%%%%%%%%%%%%%%%%%%%%%%%%%%%%%%%%%%section8.1%%%%%%%%%%%%%%%%%%%%%%%%%%%%%%%%%%%%%%%%%%%%%%%%%%%%%%%%%%%%%%%%%%%%%%%%%%
\subsection{Asymptotics for each case}

\begin{prop}\label{lifshitz-int}
Under the settings in Theorem~\ref{leading term simplified}, 
there exist $c_1, c_2\in (0,\infty )$ depending on $d, \theta$ and $u$ such that for any $p \ge 1$, 
$$
-c_1 p^{(d+\mu\theta)/(d+2+\mu\theta)} \le t^{-1}r^2 \log \E[\mass ^p] 
\le -c_2 p^{(d+\mu\theta)/(d+2+\mu\theta)}
$$
holds for sufficiently large $t$, uniformly in $x_0 \in \Lambda_1$, where we take $\mu=1$ in the case $d=1$.  
\end{prop}

\begin{proof}
We first assume $d\ge 3$ and $\alpha >d+2$.
The same argument as in Section \ref{MEO}, using the scaling with factor $s=(pt)^{1/(d+2+\mu\theta)}$ 
instead of $r=t^{1/(d+2+\mu\theta)}$ in \eqref{upper-spec} and \eqref{hk-factor}, yields 
\begin{equation}
\begin{split}
\log \E[\mass ^p]\sim 
-& (pt)^{(d+\mu\theta)/(d+2+\mu\theta)} \\
& \times \inf_{(R_s,\, \zeta) \in \S_s}\Biggl\{ \lambda_{\zeta}^s(R_s)
 +s^{(1-\mu )\theta}\sum_{q \in (s[R_s: \,l])\cap \Z ^d} s^{-d}\Bigl|\frac{\zeta_q}{s}\Bigr|^{\theta}\Biggr\} 
\end{split} \label{p-asy1}
\end{equation}
as $t\to \infty$ for any $l$. Since we know 
\begin{equation}
\begin{split}
0 < &\varliminf_{s \to \infty} \inf_{(R_s,\, \zeta) \in \S_s}\Biggl\{ \lambda_{\zeta}^s(R_s)
 +s^{(1-\mu )\theta}\sum_{q \in (s[R_s: \,l])\cap \Z ^d} s^{-d}\Bigl|\frac{\zeta_q}{s}\Bigr|^{\theta}\Biggr\}\\ 
 \le &\varlimsup_{s \to \infty} \inf_{(R_s,\, \zeta) \in \S_s}\Biggl\{ \lambda_{\zeta}^s(R_s)
 +s^{(1-\mu )\theta}\sum_{q \in (s[R_s: \,l])\cap \Z ^d} s^{-d}\Bigl|\frac{\zeta_q}{s}\Bigr|^{\theta}\Biggr\} 
 < \infty
\end{split} \label{si-lim}
\end{equation}
from Theorems \ref{F-K} and \ref{leading term}, the proof is completed. 
The other cases can be treated exactly in the same way. 
\end{proof}

\begin{rem}
If $-\lim_{t \to \infty} t^{-1}r^2 \log \E[\mass ]$ exists under the setting of the last proposition, 
denoting it by $L$, we have 
\begin{equation}
	t^{-1}r^2 \log \E[\mass ^p] \sim -L p^{(d+\mu\theta)/(d+2+\mu\theta)}.\label{moment}
\end{equation}
Indeed, when $d \ge 2$ and $\alpha > d+2$, the existence of the above limit implies 
\begin{equation*}
 \lim_{t \to \infty} \inf_{(R_r,\, \zeta) \in \S_r}\Biggl\{ \lambda_{\zeta}^r(R_r)
 +\gamma (r)^{\theta}\sum_{q \in (r[R_r: \,l])\cap \Z ^d} r^{-d}\Bigl|\frac{\zeta_q}{r}\Bigr|^{\theta}\Biggr\} =L
\end{equation*}
by Theorem \ref{leading term} and then \eqref{moment} is obvious from 
the proof of the last proposition. 
When $\alpha = d+2$, we know only that the superior limit and the inferior limit in \eqref{si-lim} tend to $L$ as $l\to \infty$.
This is still enough to show \eqref{moment}.
\end{rem}

The above remark actually applies for the case $d = 1$ and $\alpha > 3$: 
\begin{prop}\label{l1p-asy}
Under the conditions of Theorem~\ref{F-K}-$(i)$ with $\alpha > 3$, we have
 \begin{equation}
  \lim_{t\uparrow \infty} t^{-(1+\theta )/(3+\theta )}\log \E[\mass ^p]
  = -\frac{3+\theta}{1+\theta}\Big( \frac{p\pi ^2}8 \Big) ^{(1+\theta )/(3+\theta )} \label{l1p-asy-1} 
 \end{equation}
for any $p\ge 1$, uniformly in $x_0 \in \Lambda_1$.
\end{prop}

\begin{proof}
As in the proof of the last proposition we have
\begin{equation}
  \begin{split}
\log \E[\mass ^p]\sim
-&(pt)^{(1+\theta)/(3+\theta)} \\
& \times \inf_{(R_s,\, \zeta) \in \S_s}\Biggl\{ \lambda_{\zeta}^s((m/s,n/s))
 +\sum_{q \in (m-ls,n+ls)\cap \Z} s^{-1}\Bigl|\frac{\zeta_q}{s}\Bigr|^{\theta}\Biggr\} 
   \end{split} \label{l1p-asy1}
\end{equation}
as $t\to \infty$ for any $l$ in the notations of Subsection 2.3, where $s=(pt)^{1/(3+\theta)}$. 
When $\alpha >3$, we know the limit
 \begin{equation*}
  \begin{split}
& \lim_{s \to \infty} \inf_{(R_s,\, \zeta) \in \S_s}\Biggl\{ \lambda_{\zeta}^s((m/s,n/s))
 +\sum_{q \in (m-ls,n+ls)\cap \Z} s^{-1}\Bigl|\frac{\zeta_q}{s}\Bigr|^{\theta}\Biggr\} \\
 & =\frac{3+\theta}{1+\theta}\Big( \frac{\pi ^2}8 \Big) ^{(1+\theta )/(3+\theta )}.
   \end{split} 
 \end{equation*}
\end{proof}

\begin{prop}\label{pp-asy}
Under the conditions of Theorem~\ref{F-K} with $\alpha <d+2$, we have
 \begin{equation}
  \lim_{t\uparrow \infty} t^{-(d+\theta )/(\alpha +\theta )}\log \E[\mass ^p]= -p^{(d+\theta )/(\alpha +\theta )}c(d,\alpha ,\theta ,C_0) \label{pp-asy-1} 
 \end{equation}
for any $p\ge 1$, uniformly in $x_0 \in \Lambda_1$.
\end{prop}

\begin{proof} We have only to show
$$\lim_{t\uparrow \infty} t^{-(d+\theta )/(\alpha +\theta )}\log \E[\mass ^p]= -\int_{\R ^d}dq \inf_{y\in \R ^d}\Big( \frac{pC_0}{|q+y|^{\alpha}}+|y|^{\theta}\Big) .$$
The upper estimate is easy since we have
 \begin{equation*}
  \begin{split}
\E[\mass ^p] \le & \E \biggl[ E_{x_0} \biggl[ \exp \biggl\{-\int_0^t V_{\xi}(B_s) ds \biggr\} \biggr] ^p \biggr] \\
\le & \E \otimes E_{x_0} \biggl[ \exp \biggl\{-p\int_0^t V_{\xi}(B_s) ds \biggr\} \biggr]
  \end{split} 
 \end{equation*}
by removing the Dirichlet condition and using the H\"older inequality. 
For the lower estimate, we take $R$, $R_1$ and $\beta$ as in the proof of Proposition~2.2 in~\cite{FU09a} 
and restrict the integral as 
 \begin{equation*}
  \begin{split}
\E[\mass ^p]\ge \E \biggl[ E_{x_0} \biggl[& \exp \biggl\{-\int_0^t V_{\xi}(B_s) ds \biggr\} \\
 & : B_s \in \Lambda _R\text{ for }0\le s \le t\biggr] ^p : \Xi _t \biggr]
  \end{split} 
 \end{equation*}
for $t^{\beta}\ge 2(R_1+R\sqrt{d})$, where $\Xi _t$ is the set of configurations defined by
$$\{ |\xi _q|\le |q|/2\text{ for }|q|\ge t^{\beta},\text{ and }|q+\xi _q|\ge R_1+R\sqrt{d}\text{ for }|q|<t^{\beta} \} .$$
The right hand side is bounded from below by
$$\E \biggl[ \exp \biggl\{-pt \sup_{y\in \Lambda _R}V_{\xi}(y) \biggr\} : \Xi _t \biggr]
\exp (-cptR^{-2}).$$
This is estimated by the same method as in our proof of Proposition~2.2 in~\cite{FU09a}. 
\end{proof}

\begin{prop}\label{pn-asy}
Under the conditions of Theorem \ref{F-K} with $u \le 0$, we have
$$\lim_{t\uparrow \infty} t^{-(1+d/\theta )}\log \E[\mass ^p]
 =c_-(d,\theta ,pu(0))$$
for any $p\ge 1$, uniformly in $x_0 \in \Lambda_1$.
\end{prop}

\begin{proof} 
The upper and lower estimates are obtained by similar ways to the proof of Proposition \ref{pp-asy}
and that of \eqref{F-K-4} respectively.
\end{proof}

%%%%%%%%%%%%%%%%%%%%%%%%%%%%%%%%%%%%%%%%%%%%%%%%%%section8.2%%%%%%%%%%%%%%%%%%%%%%%%%%%%%%%%%%%%%%%%%%%%%%%%%%%%%%%%%%%%%%%%%%%%%%%%%%
\subsection{Intermittency}\label{Intermittency}
The initial value problem of the form \eqref{PAM} is called the ``parabolic Anderson problem'' 
in literature, see e.g.~a survey article by G\"artner and K\"onig~\cite{GK04}. 
For a wide class of random potentials, it is believed that the solution of parabolic Anderson problem 
consists of high peaks which are far from each other. 
A manifestation of this phenomenon formulated by G\"artner and Molchanov~\cite{GM90} 
is so-called ``intermittency'' defined by 
\begin{equation}
 \frac{\E [\mass ^{p_2}]^{1/p_2}} {\E [\mass ^{p_1}]^{1/p_1}}
 \stackrel{t \to \infty}{\longrightarrow} \infty \quad \textrm{for } p_1 < p_2.
 \label{intermittency}
\end{equation}
Although \eqref{intermittency} implies the concentration of $\mass$ in the $\xi$-space, 
there is a way to relate this to the spatial concentration of the solution through the ergodic theorem. 
See Sect.~1.3 of~\cite{GK04} for this point. 
G\"artner and Molchanov also proved in ~\cite{GM90} that the intermittency holds for a quite general 
class of potentials. In particular, if we consider a slightly different moment
\begin{equation} 
\E \left[\int_{\Lambda_1} \mass ^p dx_0 \right]\label{stationary-moment}
\end{equation}
in our model, then the intermittency follows by the same argument as for Theorem 3.2 of \cite{GM90}. 

Our main result Theorem~\ref{leading term simplified} gives a more detailed description of 
the concentration in the configuration space. 
Indeed, it says that the main contribution to $\E[\mass]$ comes only from minimizers of the right hand side 
of \eqref{F-K upper simplified}. 
Furthermore, we can derive the rates of the divergence in \eqref{intermittency} from the results 
in the previous subsection as follows: 
\begin{enumerate}
 \item{Under the settings in Theorem~\ref{leading term simplified}, we have
\begin{equation*}
   \frac{\E[\mass ^{p_2}]^{1/p_2}} {\E[\mass ^{p_1}]^{1/p_1}}
   \begin{cases}
   \ge \exp\set{tr^{-2}\left(c_2 p_1^{-2/(d+2+\mu\theta)}-c_1 p_2^{-2/(d+2+\mu\theta)}\right)} \\[8pt]
   \le \exp\set{tr^{-2}\left(c_1 p_1^{-2/(d+2+\mu\theta)}-c_2 p_2^{-2/(d+2+\mu\theta)}\right)}, 
   \end{cases}
\end{equation*}
for sufficiently large $t$, where $\infty >c_1\ge c_2>0$ are the constants in Proposition \ref{lifshitz-int}. }
 \item{Under the settings in Theorem~\ref{F-K} with $d=1$ and $\alpha>3$, 
it holds that
\begin{equation*}
 \begin{split}
  & \frac{\E[\mass ^{p_2}]^{1/p_2}} {\E[\mass ^{p_1}]^{1/p_1}} \\
  =& \exp\set{\frac{3+\theta}{1+\theta}\Big( \frac{\pi ^2t}8 \Big) ^{(1+\theta )/(3+\theta )} 
 \left(p_1^{-2/(3+\theta)}-p_2^{-2/(3+\theta)}+o(1) \right)}
 \end{split}
\end{equation*}
as $t$ goes to $\infty$.}
 \item{Under the settings in Theorem~\ref{F-K} with $\alpha <d+2$, 
 it holds that
 \begin{equation*}
  \begin{split}
 & \frac{\E[\mass ^{p_2}]^{1/p_2}} {\E[\mass ^{p_1}]^{1/p_1}}\\
  =& \exp\set{c(d,\alpha ,\theta ,C_0) t^{(d+\theta)/(\alpha+\theta)}
 \left(p_1^{(d-\alpha)/(\alpha+\theta)}-p_2^{{(d-\alpha)/(\alpha+\theta)}}+o(1) \right)} 
  \end{split}
\end{equation*}
as $t$ goes to $\infty$.}
 \item{Under the settings in Theorem \ref{F-K} with $u \le 0$, 
 it holds that
\begin{equation*}
  \frac{\E[\mass ^{p_2}]^{1/p_2}} {\E[\mass ^{p_1}]^{1/p_1}}
  = \exp\set{c_-(d,\theta ,u(0)) t^{1+d/\theta} \left( p_2^{d/\theta}-p_1^{d/\theta}+o(1) \right)}
\end{equation*}
as $t$ goes to $\infty$.}
\end{enumerate}
Note that in the first case, the left hand side goes to infinity only when $p_2/p_1$ is sufficiently large. 
On the other hand, the left hand sides go to infinity for any $p_2/p_1>1$ in other cases. 
This is slightly better than Theorem 3.2 of \cite{GM90} where $p_2 \ge 2$ is required. 
Note also that all these estimates hold uniformly in $x_0 \in \Lambda_1$ and therefore, 
the same estimates hold for \eqref{stationary-moment} as well. 

\newcommand{\noop}[1]{}


\begin{thebibliography}{1}

\bibitem{BLS07}
J.~Baker, M.~Loss, and G.~Stolz.
\newblock Minimizing the ground state energy of an electron in a randomly
  deformed lattice.
\newblock {\em Comm. Math. Phys.}, 283(2):397--415, 2008.

\bibitem{BLS08}
J.~Baker, M.~Loss, and G.~Stolz.
\newblock Low energy properties of the random displacement model.
\newblock {\em J. Funct. Anal.}, 256(8):2725--2740, 2009.

\bibitem{Fuk09a}
R.~Fukushima.
\newblock Brownian survival and {L}ifshitz tail in perturbed lattice disorder.
\newblock {\em J. Funct. Anal.}, 256(9):2867--2893, 2009.

\bibitem{FU09a}
R.~Fukushima and N.~Ueki.
\newblock Classical and quantum behavior of the integrated density of states
  for a randomly perturbed lattice.
\newblock {\em Annales Henri Poincar\'{e}}, 11:1053--1083, 2010.

\bibitem{GK04}
J.~G{\"a}rtner and W.~K{\"o}nig.
\newblock The parabolic {A}nderson model.
\newblock In {\em Interacting stochastic systems}, pages 153--179. Springer,
  Berlin, 2005.

\bibitem{GM90}
J.~G{\"a}rtner and S.~A. Molchanov.
\newblock Parabolic problems for the {A}nderson model. {I}. {I}ntermittency and
  related topics.
\newblock {\em Comm. Math. Phys.}, 132(3):613--655, 1990.

\bibitem{GK09}
F.~Ghribi and F.~Klopp.
\newblock Localization for the random displacement model at weak disorder.
\newblock {\em Annales Henri Poincar\'e}, 11:127--149, 2010.

\bibitem{Szn98}
A.-S. Sznitman.
\newblock {\em Brownian motion, obstacles and random media}.
\newblock Springer Monographs in Mathematics. Springer-Verlag, Berlin, 1998.

\end{thebibliography}
\end{document}